\documentclass[preprint,12pt]{elsarticle}

\usepackage[T1]{fontenc}
\usepackage[utf8]{inputenc}
\usepackage{lmodern}
\usepackage{amsmath,amssymb,amsthm,mathtools}
\usepackage{microtype}
\usepackage[hidelinks]{hyperref}

\journal{Journal of Pure and Applied Algebra}
\biboptions{numbers,sort&compress}

\newcommand{\DMeff}{\mathbf{DM}^{\mathrm{eff}}}
\newcommand{\DMeffR}{\mathbf{DM}^{\mathrm{eff}}_R}

\newcommand{\DMo}{\mathbf{DM}^{\circ}}
\newcommand{\HI}{\mathbf{HI}}
\newcommand{\HIR}{\mathbf{HI}_R}
\newcommand{\HIo}{\mathbf{HI}^{\circ}}
\newcommand{\HIoR}{\mathbf{HI}^{\circ}_R}
\newcommand{\Ab}{\mathbf{Ab}}
\newcommand{\Sm}{\mathbf{Sm}}
\newcommand{\SmPr}{\mathbf{SmPr}}
\newcommand{\Rnr}{R_{\mathrm{nr}}}
\newcommand{\Gm}{\mathbb G_m}
\newcommand{\Z}{\mathbb Z}
\newcommand{\Q}{\mathbb Q}
\newcommand{\CH}{\operatorname{CH}}
\newcommand{\Pic}{\operatorname{Pic}}
\newcommand{\NS}{\operatorname{NS}}
\newcommand{\Griff}{\operatorname{Griff}}
\newcommand{\Hom}{\operatorname{Hom}}

\newcommand{\Coker}{\operatorname{Coker}}
\newcommand{\Ker}{\operatorname{Ker}}
\newcommand{\im}{\operatorname{Im}}

\newcommand{\id}{\operatorname{id}}
\newcommand{\ulHom}{\underline{\operatorname{Hom}}}
\newcommand{\Nis}{\mathrm{Nis}}
\newcommand{\tors}{\mathrm{tors}}
\newcommand{\calT}{\mathcal T}
\newcommand{\calA}{\mathcal A}

\newcommand{\bbone}{\mathbf 1}

\newcommand{\DMoR}{\mathbf{DM}^{\circ}_R}
\newcommand{\Spec}{\operatorname{Spec}}
\newcommand{\Tor}{\operatorname{Tor}}
\newcommand{\Choweff}{\operatorname{Chow}^{\mathrm{eff}}}
\newcommand{\Chowbir}{\operatorname{Chow}^{\mathrm{bir}}}
\newcommand{\Chownor}{\operatorname{Chow}^{\mathrm{nor}}}
\newcommand{\heff}{h^{\mathrm{eff}}}
\newcommand{\hbir}{h^{\mathrm{bir}}}
\newcommand{\hnor}{h^{\mathrm{nor}}}

\theoremstyle{plain}
\newtheorem{theorem}{Theorem}[section]
\newtheorem{proposition}[theorem]{Proposition}
\newtheorem{lemma}[theorem]{Lemma}
\newtheorem{corollary}[theorem]{Corollary}
\theoremstyle{definition}
\newtheorem{definition}[theorem]{Definition}
\newtheorem{assumption}[theorem]{Assumption}
\theoremstyle{remark}
\newtheorem{remark}[theorem]{Remark}

\begin{document}

\begin{frontmatter}

\title{A coproduct obstruction for derived unramified cohomology}

\author{David Kumallagov}

\begin{abstract}
Let \(k\) be a perfect field of exponential characteristic \(p\), and let
\(R\) be a commutative \(\Z[1/p]\)-algebra.  We prove that the first derived
unramified functor
\[
        F\longmapsto R^1_{\mathrm{nr},R}F
\]
from homotopy invariant Nisnevich sheaves with transfers of \(R\)-modules to
birational sheaves commutes with arbitrary small direct sums.  This gives a positive answer, after inverting the exponential characteristic, to a question of Kahn and
Sujatha; on smooth projective varieties no inversion is needed.

We also describe an obstruction to this for the functor $R^2_{\mathrm{nr},R}$ in categorical terms, which includes the familiar Griffiths group obstruction.  As applications of the motivic nature of the functors \(R^q_{\mathrm{nr}}\), we prove
torsion-order bounds and a correspondence-detection statement for surfaces.
\end{abstract}

\begin{keyword}
derived unramified cohomology \sep birational motives \sep triangulated categories \sep homotopy invariant sheaves with transfers \sep Griffiths groups \sep torsion order
\MSC[2020] 14F42 \sep 18G80 \sep 14C15 \sep 14E05
\end{keyword}

\end{frontmatter}

\section{Introduction}

Let \(k\) be a perfect field.  We write \(\DMeff(k)\) for Voevodsky's
triangulated category of effective motives with transfers and \(\HI(k)\) for the
heart of its homotopy \(t\)-structure, i.e. the category of homotopy invariant
Nisnevich sheaves with transfers.  If \(R\) is a commutative coefficient ring,
we write \(\DMeffR(k)\) and \(\HIR(k)\) for the corresponding \(R\)-linear
categories.  When no confusion is possible we omit \(k\) from the notation.

Kahn and Sujatha constructed the triangulated category \(\DMo\) of birational
motivic complexes and a fully faithful functor
\[
        i^\circ:\DMo\hookrightarrow \DMeff
\]
with right adjoint \(\Rnr\) \cite{KahnSujathaTriangulated,KahnSujathaDerived}.
For \(F\in\HI\), the derived unramified functors are
\[
        R^q_{\mathrm{nr}}F=H^q(\Rnr(F[0]))\in\HIo,
\]
where \(\HIo\) denotes the category of birational homotopy invariant sheaves.
We use the standard \(R\)-linear form of these constructions; see
\cite[Theorem~2.2.3]{BondarkoChowPure} and
\cite[Section~5.2]{BondarkoKumallagov}.

In \cite[Section 8, Question 4]{KahnSujathaDerived}, Kahn and Sujatha observe that
\(R^0_{\mathrm{nr}}\) preserves arbitrary direct sums, while
\(R^2_{\mathrm{nr}}\) does not preserve them in general, and ask whether
\(R^1_{\mathrm{nr}}\) preserves arbitrary direct sums.  The main result of the
present paper is the following positive answer in the prime-to-characteristic
coefficient range.

\begin{theorem}\label{thm:intro-main}
Let \(k\) be a perfect field of exponential characteristic \(p\), and let
\(R\) be a commutative \(\Z[1/p]\)-algebra.  Then
\[
        R^1_{\mathrm{nr},R}:\HIR(k)\longrightarrow\HIoR(k)
\]
preserves arbitrary small direct sums.  
\end{theorem}

The first step is a pointwise theorem on smooth projective varieties.  This part
does not require the assumption \(R\supset\Z[1/p]\).

\begin{theorem}\label{thm:intro-pointwise}
Let \(R\) be any commutative ring and let \(X/k\) be smooth projective.  For
every small family \((F_i)_{i\in I}\) in \(\HIR\), the canonical morphism
\[
        \bigoplus_{i\in I} R^1_{\mathrm{nr},R}F_i(X)
        \longrightarrow
        R^1_{\mathrm{nr},R}\left(\bigoplus_{i\in I}F_i\right)(X)
\]
is an isomorphism.
\end{theorem}

The proof of Theorem \ref{thm:intro-pointwise} is based on a purely algebraic
obstruction statement.  Given a distinguished triangle in the triangulated category $\mathcal T$
\[
        A\longrightarrow P\longrightarrow B\longrightarrow A[1]
\]
with \(P\) compact, Theorem \ref{thm:formal-obstruction} measures the failure of
\(E\mapsto\mathcal T(B,E[q])\) to preserve direct sums on the heart of a
\(t\)-structure by a cokernel involving \(E\mapsto\mathcal T(A,E[q-1])\).  If
\(A\) is \(c\)-negative with respect to the \(t\)-structure, then the obstruction
vanishes for \(1\le q\le c\).

For the distinguished triangle
\[
        \nu^{\ge1}M(X)\longrightarrow M(X)
        \longrightarrow i^\circ\nu_{\le0}M(X)
        \longrightarrow \nu^{\ge1}M(X)[1],
\]
\cite[Proposition 2.3 and proof of Theorem 4.1]{KahnSujathaDerived} imply, for smooth projective
\(X\), that
\[
        \nu^{\ge1}M(X)\in \DMeff_{\le -1}.
\]
Hence the degree-one obstruction vanishes.  This proves Theorem
\ref{thm:intro-pointwise}.

The passage from smooth projective varieties to all smooth varieties uses
Bondarko's criterion for birational sheaves
\cite[Lemma 3.3.3(II)(3)]{BondarkoResolution}: over a perfect field of
exponential characteristic \(p\), after inverting \(p\), a morphism between
birational homotopy invariant sheaves with transfers is an isomorphism if and
only if it is an isomorphism on all smooth projective varieties.  This criterion
upgrades Theorem \ref{thm:intro-pointwise} to Theorem
\ref{thm:intro-main}.

The same obstruction formula also explains why degree two behaves differently.
For a smooth projective connected variety over an algebraically closed field,
Kahn--Sujatha compute \(R^2_{\mathrm{nr}}\mathbb G_m(X)\) through a short exact
sequence involving \(\operatorname{Hom}(\operatorname{Griff}_1(X),\Z)\).  More
generally, for torsion-free \(A\), the corresponding term is
\(\operatorname{Hom}(\operatorname{Griff}_1(X),A)\).  Since this functor does not
commute with infinite direct sums when \(\operatorname{Griff}_1(X)\otimes\Q\) is
infinite-dimensional, \(R^2_{\mathrm{nr}}\) fails to preserve direct sums in
general.  We also spell out a related Bockstein computation, stated without
proof in the introduction of Kahn--Sujatha, which gives a short exact sequence
for \(R^1_{\mathrm{nr}}(\Gm/m)\) in terms of \(\NS(X)_{\tors}\) and the
\(m\)-torsion of \(D^1(X)\).

Finally, we record two consequences which are independent of the proof of the
main theorem but help situate the functors \(R^q_{\mathrm{nr}}\).  First, for every
\(q>0\) the functor \(X\mapsto R^q_{\mathrm{nr}}F(X)\) on smooth projective
varieties is a normalized birational motivic invariant.
Kahn's torsion-order theorem therefore gives uniform annihilators for these
groups on varieties with a decomposition of the diagonal.  Second, the theorem
of Sato and Yamazaki on torsion birational motives of surfaces implies that, for
such surfaces, the action of correspondences on all \(R^q_{\mathrm{nr}}F\) is
detected by the first two classical unramified cohomology functors.

\section{Conventions and basic definitions}
\label{sec:conventions}

All coproducts in this paper are small.  If \(\mathcal T\) is a triangulated
category with small coproducts, an object \(P\in\mathcal T\) is called
\emph{compact} if the functor
\[
        \mathcal T(P,-):\mathcal T\longrightarrow \Ab
\]
commutes with small coproducts.

Let \(k\) be a perfect field.  We denote by \(\Sm_k\) the category of smooth varieties over \(k\), and by \(\SmPr(k)\) the full subcategory of
smooth projective \(k\)-varieties.  If
\(R\) is a commutative ring, \(\DMeffR(k)\) denotes Voevodsky's category of
\(R\)-linear effective motives.  We write \(M(X)_R\) for the motive of
\(X\in\Sm_k\).  The homotopy
\(t\)-structure on \(\DMeffR(k)\) has heart \(\HIR(k)\), the abelian category of
homotopy invariant Nisnevich sheaves with transfers of \(R\)-modules
\cite[Section 3.1]{VoevodskyMotives}.  We view an object
\(F\in\HIR(k)\) as the complex \(F[0]\) in \(\DMeffR(k)\).

A sheaf \(S\in\HIR(k)\) is \emph{birational} if, for every dense open immersion
\(U\subset X\) in \(\Sm_k\), the restriction map
\[
        S(X)\longrightarrow S(U)
\]
is an isomorphism.  Equivalently, \(S_{-1}=0\), where \((-)_ {-1}\) denotes
Voevodsky's contraction; see 
\cite[Proposition 1.4]{KahnSujathaDerived} or \cite[Proposition 3.2.1]{BondarkoKumallagov}.  We write \(\HIoR(k)\) for the full
subcategory of birational sheaves.  It is a Serre subcategory of \(\HIR(k)\), and
its inclusion has a right adjoint \(R^0_{\rm nr,R}\), see \cite[Lemma A.3.1]{KahnSujathaTriangulated},  \cite[Theorem 2.2.3]{BondarkoChowPure} or \cite[Theorem 4.2.6]{BondarkoKumallagov}.

Let \(\DMoR(k)\) denote the \(R\)-linear triangulated category of birational motives, obtained as the Verdier quotient of \(\DMeffR(k)\) by the
localizing subcategory \(\DMeffR(k)(1)\), see \cite[Definition 4.2.1 and Theorem 4.2.2]{KahnSujathaTriangulated}.  Kahn and Sujatha
construct adjunctions
\[
        \nu_{\le0}:\DMeffR(k)\rightleftarrows\DMoR(k):i^\circ,
        \qquad
        i^\circ:\DMoR(k)\rightleftarrows\DMeffR(k):R_{\rm nr,R},
\]
with \(\nu_{\le0}\dashv i^\circ\dashv R_{\rm nr,R}\), where
\(i^\circ\) is fully faithful \cite[Theorem 4.2.2]{KahnSujathaTriangulated};
for the \(R\)-linear form used here, see also 
\cite[Theorem 2.2.3]{BondarkoChowPure}.  We often identify \(\DMoR(k)\) with the
essential image of \(i^\circ\).  For \(F\in\HIR(k)\), the
\emph{derived unramified functors} are
\[
        R^q_{\rm nr,R}F:=H^q\bigl(R_{\rm nr,R}(F[0])\bigr)\in\HIoR(k).
\]
When \(R=\mathbb Z\), we omit \(R\) from the notation.

\section{A coproduct obstruction theorem}

This section is purely formal, and contains the algebraic core of the argument.

\begin{assumption}\label{ass:formal-setup}
Let \(\calT\) be a triangulated category with small coproducts and a
\(t\)-structure.  Let \(\calA\) be its heart.  We assume that small coproducts of
objects of \(\calA\), computed in \(\calT\), again belong to \(\calA\) and
coincide with coproducts in \(\calA\).  Let
\begin{equation}\label{eq:formal-triangle}
        A\longrightarrow P\longrightarrow B\longrightarrow A[1]
\end{equation}
be a distinguished triangle in \(\calT\), and assume that \(P\) is compact.
\end{assumption}

For \(E\in\calA\) and \(q\in\Z\), put
\[
        H^q(E)=\calT(P,E[q]),\qquad
        T^q(E)=\calT(A,E[q]),\qquad
        R^q(E)=\calT(B,E[q]).
\]
The triangle \eqref{eq:formal-triangle} gives a long exact sequence
\begin{equation}\label{eq:formal-long}
\cdots\to H^{q-1}(E)\xrightarrow{\alpha_E^{q-1}}T^{q-1}(E)
\to R^q(E)
\to H^q(E)\xrightarrow{\alpha_E^q}T^q(E)\to\cdots .
\end{equation}
Set
\[
        C^{q-1}(E)=\Coker(\alpha_E^{q-1}),\qquad
        K^q(E)=\Ker(\alpha_E^q).
\]
Then \eqref{eq:formal-long} contains a natural short exact sequence
\begin{equation}\label{eq:formal-short}
        0\to C^{q-1}(E)\to R^q(E)\to K^q(E)\to 0.
\end{equation}

\begin{lemma}\label{lem:finite-support}
Let \(\mathcal F,\mathcal G:\calA\to\Ab\) be additive functors and let
\(\delta:\mathcal F\to\mathcal G\) be a natural transformation.  If \(\mathcal F\)
preserves small coproducts, then the functor
\[
        E\longmapsto\Ker(\delta_E)
\]
preserves small coproducts.
\end{lemma}

\begin{proof}
This is a standard result.
\end{proof}

\begin{theorem}\label{thm:formal-obstruction}
Under Assumption \ref{ass:formal-setup}, for every small family \((E_i)_{i\in I}\)
in \(\calA\) and every \(q\in\Z\) there is a natural short exact sequence
\begin{multline}\label{eq:formal-obstruction}
0\longrightarrow
\bigoplus_{i\in I}R^q(E_i)
\longrightarrow
R^q\biggl(\bigoplus_{i\in I}E_i\biggr)\longrightarrow\\
\Coker\biggl[
\bigoplus_{i\in I}T^{q-1}(E_i)
\longrightarrow
T^{q-1}\biggl(\bigoplus_{i\in I}E_i\biggr)
\biggr]
\longrightarrow 0 .
\end{multline}
\end{theorem}

\begin{proof}
Put \(E=\bigoplus_iE_i\).  Since \(P\) is compact, the functors
\(H^m(E)=\calT(P,E[m])\) preserve small coproducts.  By Lemma
\ref{lem:finite-support}, applied to \(\alpha^q:H^q\to T^q\), the canonical map
\begin{equation}\label{eq:K-preserves}
        \bigoplus_iK^q(E_i)\longrightarrow K^q(E)
\end{equation}
is an isomorphism.

Let
\[
        \eta:\bigoplus_iT^{q-1}(E_i)\longrightarrow T^{q-1}(E)
\]
be the canonical map; it is obviously injective.  Since \(H^{q-1}\) preserves
coproducts, the image of
\[
        H^{q-1}(E)\longrightarrow T^{q-1}(E)
\]
is exactly
\[
        \eta\biggl(\bigoplus_i\im\bigl(H^{q-1}(E_i)\to T^{q-1}(E_i)\bigr)\biggr).
\]
Consequently the canonical map
\begin{equation}\label{eq:C-obstruction}
        \bigoplus_i C^{q-1}(E_i)\longrightarrow C^{q-1}(E)
\end{equation}
is injective and has cokernel \(\Coker(\eta)\).

Now form the diagram of short exact sequences obtained from
\eqref{eq:formal-short} for the family \((E_i)\) and for \(E\):
\[
\begin{array}{ccccccccc}
0&\to&\bigoplus_iC^{q-1}(E_i)&\to&\bigoplus_iR^q(E_i)&\to&\bigoplus_iK^q(E_i)&\to&0\\
&&\downarrow&&\downarrow&&{}^{\sim}\!\downarrow&&\\
0&\to&C^{q-1}(E)&\to&R^q(E)&\to&K^q(E)&\to&0 .
\end{array}
\]
The right vertical arrow is an isomorphism by \eqref{eq:K-preserves}, and the
left vertical arrow has cokernel \(\Coker(\eta)\) by \eqref{eq:C-obstruction}.
The snake lemma gives the asserted short exact sequence.
\end{proof}

\begin{corollary}\label{cor:negative}
Assume \ref{ass:formal-setup}.  If \(A\in\calT_{\le -c}\) for some \(c\ge1\), then
for every \(1\le q\le c\) the functor
\[
        E\longmapsto R^q(E)=\calT(B,E[q])
\]
preserves small coproducts on \(\calA\).
\end{corollary}

\begin{proof}
For \(E\in\calA\) and \(1\le q\le c\), one has
\[
        E[q-1]\in\calT_{\ge -(q-1)}\subseteq \calT_{\ge -c+1}.
\]
The orthogonality axiom of the \(t\)-structure gives
\[
        T^{q-1}(E)=\calT(A,E[q-1])=0.
\]
Thus the obstruction term in Theorem \ref{thm:formal-obstruction} is zero.
\end{proof}

\section{The Kahn--Sujatha triangle}

We now apply the formal result to motives.  In this section \(R\) is an
arbitrary commutative ring and all motivic categories are \(R\)-linear.

For \(M\in\DMeffR\), let us write
\[
        \nu^{\ge1}M=\ulHom(R(1),M)(1)
\]
and a functorial triangle \cite[Proposition 1.2]{KahnSujathaDerived}
\begin{equation}\label{eq:KS-triangle}
        \nu^{\ge1}M\longrightarrow M\longrightarrow i^\circ\nu_{\le0}M
        \longrightarrow \nu^{\ge1}M[1].
\end{equation}
For \(C\in\DMeffR\), this triangle gives the long exact sequence
\cite[Proposition 2.2]{KahnSujathaDerived}
\begin{multline}\label{eq:KS-long}
\cdots\to H^n(X,R_{\rm nr,R} C)\to H^n(X,C)\to
\DMeffR(\nu^{\ge1}M(X)_R,C[n])\\
\to H^{n+1}(X,R_{\rm nr,R} C)\to\cdots .
\end{multline}
If \(C=F[0]\) with \(F\in\HIR\), then the hypercohomology spectral sequence for
\(R_{\rm nr,R}(F[0])\) degenerates because birational sheaves have no higher Nisnevich
cohomology \cite[Lemma 2.1]{KahnSujathaDerived}.  Hence
\[
        H^n(X,R_{\rm nr,R}(F[0]))\simeq R^n_{\mathrm{nr},R}F(X).
\]

The crucial input is the following  statement.

\begin{lemma}\label{lem:nu-negative}
Let \(X/k\) be smooth projective.  Then
\[
        \nu^{\ge1}M(X)_R\in{\DMeffR}_{\le -1}
\]
for the homotopy \(t\)-structure.  Consequently, for every \(F\in\HIR\),
\[
        \DMeffR(\nu^{\ge1}M(X)_R,F[0])=0.
\]
\end{lemma}

\begin{proof}
The first statement is standard, see the proof of \cite[Theorem 4.1]{KahnSujathaDerived}.

 The final vanishing follows from
\(F[0]\in{\DMeffR}_{\ge0}\) and the orthogonality of the \(t\)-structure.
\end{proof}

\begin{proposition}\label{prop:R1-kernel}
For every smooth projective \(X/k\) and every \(F\in\HIR\), there is a natural
exact sequence
\[
        0\to R^1_{\mathrm{nr},R}F(X)\to H^1_{\Nis}(X,F)
        \xrightarrow{\alpha_{X,F}}
        \DMeffR(\nu^{\ge1}M(X)_R,F[1]).
\]
In particular,
\[
        R^1_{\mathrm{nr},R}F(X)=\Ker(\alpha_{X,F}).
\]
\end{proposition}

\begin{proof}
Apply \eqref{eq:KS-long} to \(C=F[0]\).  The term
\(\DMeffR(\nu^{\ge1}M(X)_R,F[0])\) vanishes by Lemma \ref{lem:nu-negative}, and
the preceding discussion identifies \(H^n(X,R_{\rm nr,R}(F[0]))\) with
\(R^n_{\mathrm{nr},R}F(X)\).
\end{proof}

\begin{theorem}\label{thm:Rq-obstruction}
Let \(X/k\) be smooth projective, and let \((F_i)_{i\in I}\) be a small family in
\(\HIR\).  Put \(A_X=\nu^{\ge1}M(X)_R\).  For every \(q\ge1\), there is a natural
short exact sequence
\begin{multline}\label{eq:Rq-obstruction}
0\longrightarrow
\bigoplus_iR^q_{\mathrm{nr},R}F_i(X)
\longrightarrow
R^q_{\mathrm{nr},R}\left(\bigoplus_iF_i\right)(X)
\longrightarrow\\
\Coker\left[
\bigoplus_i\DMeffR(A_X,F_i[q-1])
\longrightarrow
\DMeffR\left(A_X,\bigoplus_iF_i[q-1]\right)
\right]
\longrightarrow0.
\end{multline}
In particular, the canonical map
\[
        \bigoplus_iR^1_{\mathrm{nr},R}F_i(X)
        \longrightarrow
        R^1_{\mathrm{nr},R}\left(\bigoplus_iF_i\right)(X)
\]
is an isomorphism.
\end{theorem}

\begin{proof}
We apply Theorem \ref{thm:formal-obstruction} to the triangle
\eqref{eq:KS-triangle} for \(M=M(X)_R\).  Thus
\[
        A=A_X=\nu^{\ge1}M(X)_R,
        \qquad
        P=M(X)_R,
        \qquad
        B=i^\circ\nu_{\le0}M(X)_R .
\]
The motive \(M(X)_R\) is compact in \(\DMeffR\)
\cite[Section 3.2]{VoevodskyMotives}.  The homotopy \(t\)-structure is compatible
with small coproducts, and its heart is \(\HIR\).

For \(F_i\in\HIR\), \cite[Proposition 2.2 and Lemma 2.1]{KahnSujathaDerived} give natural
isomorphisms
\begin{equation}\label{eq:identify-B-Rnr}
        \DMeffR\bigl(i^\circ\nu_{\le0}M(X)_R,F_{i}[q]\bigr)
        \simeq R^q_{\mathrm{nr},R}F_{i}(X) .
\end{equation}

Theorem \ref{thm:formal-obstruction} therefore gives the displayed exact
sequence.  For \(q=1\), Lemma \ref{lem:nu-negative} gives
\[
        \DMeffR(A_X,F[0])=0,
\]
where \(F=\bigoplus_iF_i\).  Hence the obstruction term is
zero in degree one.
\end{proof}

The last assertion of Theorem \ref{thm:Rq-obstruction} proves Theorem
\ref{thm:intro-pointwise}.

\section{A projective criterion for birational sheaves}

The pointwise theorem just proved is a statement on smooth projective varieties.
We now use Bondarko's criterion for birational homotopy invariant sheaves to
pass from smooth projective varieties to all smooth varieties after inverting the
exponential characteristic.

Throughout this section \(k\) is perfect of exponential characteristic \(p\), and
\(R\) is a commutative \(\Z[1/p]\)-algebra.

\begin{lemma}\label{lem:evaluation-coproducts}
Let \(X\in\Sm_k\) and let \((S_i)_{i\in I}\) be a small family in \(\HIR\).  Then
\[
        \left(\bigoplus_i S_i\right)(X)\simeq \bigoplus_i S_i(X).
\]
Consequently, arbitrary small direct sums of birational sheaves in \(\HIR\) are
birational.
\end{lemma}

\begin{proof}
For a homotopy invariant sheaf with transfers \(S\), evaluation on \(X\) is
represented by the motive \(M(X)_R\) \cite[Proposition 2.1.1(5)]{BondarkoChowPure}:
\[
        S(X)=\DMeffR(M(X)_R,S[0]).
\]
and the statement follows from the compactness of \(M(X)_R\).
\end{proof}

We shall use the following consequence.

\begin{theorem}\label{thm:chow-descent}
Let
\[
        u:S\longrightarrow T
\]
be a morphism in \(\HIoR(k)\).  Suppose that
\[
        u(P):S(P)\xrightarrow{\sim}T(P)
\]
is an isomorphism for every smooth projective \(P/k\).  Then \(u(X)\) is an
isomorphism for every smooth \(X/k\).  In particular, \(u\) is an isomorphism of
birational sheaves.
\end{theorem}

\begin{proof}
For \(R=\Z[1/p]\) this is exactly 
\cite[Lemma 3.3.3(II)(3)]{BondarkoResolution}, and the case of arbitrary $R$ is proved similarly.
\end{proof}

\begin{proof}[Proof of Theorem \ref{thm:intro-main}]
Let \((F_i)_{i\in I}\) be a small family in \(\HIR\), and let
\[
        \Phi:\bigoplus_iR^1_{\mathrm{nr},R}F_i
        \longrightarrow
        R^1_{\mathrm{nr},R}\left(\bigoplus_iF_i\right)
\]
be the canonical morphism.  The right hand side is birational by construction,
and the left hand side is birational by Lemma \ref{lem:evaluation-coproducts}.

Let \(P/k\) be smooth projective.  By Lemma \ref{lem:evaluation-coproducts},
\[
        \left(\bigoplus_iR^1_{\mathrm{nr},R}F_i\right)(P)
        \simeq
        \bigoplus_iR^1_{\mathrm{nr},R}F_i(P).
\]
Under this identification, \(\Phi(P)\) is exactly the isomorphism of Theorem
\ref{thm:intro-pointwise}.  Therefore \(\Phi(P)\) is an isomorphism for all
smooth projective \(P\).  Theorem \ref{thm:chow-descent} now implies that
\(\Phi\) is an isomorphism of birational sheaves.  Evaluating on any smooth
\(X\) and using Lemma \ref{lem:evaluation-coproducts} gives the desired isomorphism in Theorem \ref{thm:intro-main}.
\end{proof}

\begin{remark}\label{rem:no-inversion}
If \(R\) is arbitrary, Theorem \ref{thm:intro-pointwise} still gives the desired
isomorphism after evaluation on every smooth projective variety.  Thus the same
sheaf-level conclusion holds over fields for which every finitely generated
function field admits a smooth projective model; in particular this recovers the
integral characteristic zero case.  
\end{remark}

\begin{remark}
Let \(n\ge0\), and let
$
DM^{n\operatorname{-bir}}
=
DM^{\operatorname{eff}}/DM^{\operatorname{eff}}\{n+1\}
$
be the \(n\)-birational category from \cite[Section 3.2]{BondarkoKumallagov}. There are adjunctions
\[
        p_{n}:\DMeffR \rightleftarrows DM^{n\operatorname{-bir}}:i^{(n)},
        \qquad
        i_{n}:\DMeffR \{n\}\rightleftarrows\DMeffR:r_{n},
\]
and set $\nu^{\ge n} = i_{n}\circ r_n$.

For \(F\in\mathbf{HI}\) and \(X\in\mathbf{SmPr}\), put
\[
\mathbb H^q_{nr,n}(X,F)
=
DM^{\operatorname{eff}}
\bigl(i^{(n)}p_nM(X),F[q]\bigr).
\]
Then the same proof as Theorem~3.3 gives a natural exact sequence

\[
\begin{aligned}
0\longrightarrow
\bigoplus_i\mathbb H^q_{\operatorname{nr},n}(X,F_i)
\longrightarrow
\mathbb H^q_{\operatorname{nr},n}\left(X,\bigoplus_iF_i\right)
\longrightarrow
\operatorname{Coker}\left[
\bigoplus_i
DM^{\operatorname{eff}}
\bigl(\nu^{\ge n+1}M(X),F_i[q-1]\bigr)
\right.
\\
\left.
\longrightarrow
DM^{\operatorname{eff}}
\bigl(\nu^{\ge n+1}M(X),\bigoplus_iF_i[q-1]\bigr)
\right]
\longrightarrow0 .
\end{aligned}
\]

Now
\(\nu^{\ge n+1}M(X)\in DM^{\operatorname{eff}}_{\le -(n+1)}\);
hence
\[
\bigoplus_i\mathbb H^q_{nr,n}(X,F_i)
\simeq
\mathbb H^q_{nr,n}\left(X,\bigoplus_iF_i\right)
\quad
(1\le q\le n+1).
\]
For \(n=0\) these total groups coincide with the groups
\(R^q_{nr}F(X)\). For \(n>0\) they are related to
the sheaves \(R^q_{nr,n}F\) by a Nisnevich hypercohomology spectral
sequence, and one should not identify them without further hypotheses.
\end{remark}

\section{The degree two obstruction}

Theorem \ref{thm:Rq-obstruction} also explains why degree two is different.  For
\(q=2\), the obstruction to preservation of direct sums by evaluation of
\(R^2_{\mathrm{nr},R}\) at a smooth projective \(X\) is
\begin{equation}\label{eq:R2-obstruction}
\Coker\left[
\bigoplus_i\DMeffR(\nu^{\ge1}M(X)_R,F_i[1])
\to
\DMeffR\left(\nu^{\ge1}M(X)_R,\bigoplus_iF_i[1]\right)
\right].
\end{equation}
There is no reason for this cokernel to vanish.

We now recall the concrete form of this obstruction in the computation of
Kahn and Sujatha for \(\Gm\).  In this section let \(k\) be algebraically closed and
let \(X/k\) be smooth projective and connected.  Recall that
\cite[Theorem 1(i)--(iii)]{KahnSujathaDerived}
\[
        R^0_{\mathrm{nr}}\Gm(X)=k^*,
        \qquad
        R^1_{\mathrm{nr}}\Gm(X)\simeq \Pic^\tau(X),
\]
and a short exact sequence
\begin{equation}\label{eq:R2-Gm}
        0\to D^1(X)\to R^2_{\mathrm{nr}}\Gm(X)
        \to \Hom(\Griff_1(X),\Z)\to0,
\end{equation}
where the notation is as follows.

\begin{definition}\label{def:cycle-notation}
Let \(\Pic^\tau(X)\subset\Pic(X)\) be the subgroup of line bundles numerically
equivalent to zero, and put
\[
        N^1(X)=\Pic(X)/\Pic^\tau(X).
\]
Let \(A^{\rm alg}_1(X)\) be the group of one-cycles on \(X\) modulo algebraic
equivalence, and let \(N_1(X)\) be the group of one-cycles modulo numerical
equivalence.  The numerical Griffiths group of one-cycles is
\[
        \Griff_1(X)=\Ker\bigl(A^{\rm alg}_1(X)\to N_1(X)\bigr).
\]
Finally,
\[
        D^1(X)=\Coker\bigl(N^1(X)\to\Hom(N_1(X),\Z)\bigr),
\]
where the map is induced by the intersection pairing between divisors and
one-cycles.
\end{definition}

More generally, for a torsion-free abelian group \(A\), set
\(F_A=\Gm\otimes A\).  Kahn and Sujatha note in Section 8 of
\cite{KahnSujathaDerived} that the same argument gives
\begin{equation}\label{eq:R2-GmA}
        0\to D^1(X)\otimes A\to R^2_{\mathrm{nr}}F_A(X)
        \to \Hom(\Griff_1(X),A)\to0.
\end{equation}

We shall use the following elementary linear algebra observation.

\begin{lemma}\label{lem:hom-vector}
Let \(V\) be an infinite-dimensional \(\Q\)-vector space.  Then
\[
        A\longmapsto \Hom_{\Q}(V,A)
\]
does not commute with arbitrary direct sums.
\end{lemma}

\begin{proposition}\label{prop:Griffiths-obstruction}
Let \(k\) be algebraically closed, let \(X/k\) be smooth projective, and let
\((A_i)_{i\in I}\) be a small family of torsion-free abelian groups.  Then the
natural map
\[
        \bigoplus_iR^2_{\mathrm{nr}}(\Gm\otimes A_i)(X)
        \longrightarrow
        R^2_{\mathrm{nr}}\left(\Gm\otimes\bigoplus_iA_i\right)(X)
\]
has cokernel naturally isomorphic to
\[
        \Coker\left[
        \bigoplus_i\Hom(\Griff_1(X),A_i)
        \longrightarrow
        \Hom\left(\Griff_1(X),\bigoplus_iA_i\right)
        \right].
\]
In particular, if \(\Griff_1(X)\otimes\Q\) is infinite-dimensional and
\(A_i=\Q\) for countably many \(i\), then the comparison map is not surjective.
\end{proposition}

\begin{proof}
Apply the natural exact sequence \eqref{eq:R2-GmA} to each \(A_i\) and to
\(\bigoplus_iA_i\).  Since tensor product commutes with direct sums,
\[
        \bigoplus_i(D^1(X)\otimes A_i)
        \xrightarrow{\sim}
        D^1(X)\otimes\bigoplus_iA_i.
\]
The map
\[
        \bigoplus_i\Hom(\Griff_1(X),A_i)\to
        \Hom\left(\Griff_1(X),\bigoplus_iA_i\right)
\]
is injective.  The snake lemma applied
to the resulting diagram of short exact sequences gives the asserted cokernel
formula.

If \(A_i=\Q\), then
\[
        \Hom_{\Z}(\Griff_1(X),\Q)
        \simeq \Hom_{\Q}(\Griff_1(X)\otimes\Q,\Q).
\]
The last assertion follows from Lemma \ref{lem:hom-vector}.
\end{proof}

\begin{remark}
Examples with \(\Griff_1(X)\otimes\Q\) infinite-dimensional are known in the
classical literature on Griffiths groups, starting with Clemens' work on
homological equivalence modulo algebraic equivalence and later refinements such
as Schoen's examples \cite{Clemens,Schoen}.  Kahn--Sujatha use such examples in
Section 8 of \cite{KahnSujathaDerived} to show that \(R^2_{\mathrm{nr}}\) does
not preserve arbitrary direct sums.  Proposition \ref{prop:Griffiths-obstruction}
records the exact cokernel responsible for this failure in the family
\(\Gm\otimes A\).
\end{remark}

\subsection{A Bockstein calculation for \texorpdfstring{$\Gm/m$}{Gm/m}}

For completeness we include the short Bockstein argument proving the calculation
of Kahn and Sujatha for \(R^1_{\mathrm{nr}}(\Gm/m)\), which is stated without proof
in their introduction \cite[p. 5]{KahnSujathaDerived}.  We keep the notation of
\eqref{eq:R2-Gm}.  
We write
\[
        {}_mD^1(X)=\Ker\bigl(m:D^1(X)\to D^1(X)\bigr)
\]
for the subgroup killed by \(m\).  

\begin{proposition}\label{prop:Gm-mod-m}
Let \(k\) be algebraically closed, let \(X/k\) be smooth projective and
connected, and let \(m>0\).  Then there is a natural short exact sequence
\[
0\longrightarrow
        \NS(X)_{\tors}/m
\longrightarrow
        R^1_{\mathrm{nr}}(\Gm/m)(X)
\longrightarrow
        {}_mD^1(X)
\longrightarrow 0 .
\]
Here \(\Gm/m\) denotes the cokernel of multiplication by \(m\) on \(\Gm\) in
\(\HI(k)\).
\end{proposition}

\begin{proof}
Let
\[
        K=\Ker([m]:\Gm\to\Gm),\qquad
        I=\im([m]:\Gm\to\Gm),\qquad
        Q=\Gm/m.
\]
Since \(\HI(k)\) is abelian, multiplication by \(m\) factors as two short exact
sequences
\begin{equation}\label{eq:Bockstein-two-ses}
        0\to K\to \Gm\to I\to0,
        \qquad
        0\to I\to \Gm\to Q\to0.
\end{equation}
The sheaf \(K\) is birational, therefore
\(R^q_{\mathrm{nr}}K=0\) for \(q>0\) \cite[Lemma 2.4]{KahnSujathaDerived}.

The first short exact sequence in \eqref{eq:Bockstein-two-ses} therefore gives
isomorphisms
\begin{equation}\label{eq:I-Gm-iso}
        R^q_{\mathrm{nr}}\Gm(X)\xrightarrow{\sim} R^q_{\mathrm{nr}}I(X)
        \qquad(q\ge1).
\end{equation}
Under these identifications, the morphism induced by the inclusion
\(I\hookrightarrow\Gm\),
\[
        R^q_{\mathrm{nr}}I(X)\longrightarrow R^q_{\mathrm{nr}}\Gm(X),
\]
is multiplication by \(m\).  This is because the composite
\(\Gm\to I\hookrightarrow\Gm\) is exactly the endomorphism \([m]\) of \(\Gm\),
and \(R^q_{\mathrm{nr}}\) is additive.

Apply the long exact cohomology sequence of the cohomological \(\delta\)-functor
\(R^\bullet_{\mathrm{nr}}\), cf. \cite[Section 2]{KahnSujathaDerived}, to the
second short exact sequence in \eqref{eq:Bockstein-two-ses}.  The relevant part is
\[
        R^1_{\mathrm{nr}}I(X)\to R^1_{\mathrm{nr}}\Gm(X)\to
        R^1_{\mathrm{nr}}Q(X)\to R^2_{\mathrm{nr}}I(X)\to
        R^2_{\mathrm{nr}}\Gm(X).
\]
Using \eqref{eq:I-Gm-iso}, this becomes the Bockstein short exact sequence
\begin{equation}\label{eq:Bockstein-short}
0\to
        R^1_{\mathrm{nr}}\Gm(X)/m
\to
        R^1_{\mathrm{nr}}(\Gm/m)(X)
\to
        {}_mR^2_{\mathrm{nr}}\Gm(X)
\to0 .
\end{equation}

We now identify the two end terms.  By 
\cite[Theorem 1(ii)]{KahnSujathaDerived},
\[
        R^1_{\mathrm{nr}}\Gm(X)\simeq \Pic^\tau(X).
\]
There is an exact sequence
\[
        0\to \Pic^0(X)\to \Pic^\tau(X)\to \NS(X)_{\tors}\to0.
\]
The group \(\Pic^0(X)=\Pic^0_{X/k}(k)\) is divisible, because \(k\) is algebraically closed.  Therefore
\[
        \Pic^\tau(X)/m\simeq \NS(X)_{\tors}/m.
\]
This identifies the left term of \eqref{eq:Bockstein-short}.

For the right term, use the short exact sequence \eqref{eq:R2-Gm}
\[
        0\to D^1(X)\to R^2_{\mathrm{nr}}\Gm(X)
        \to \Hom(\Griff_1(X),\Z)\to0.
\]
The last group is torsion-free, hence every element of \(R^2_{\mathrm{nr}}\Gm(X)\) killed by \(m\) lies in
\(D^1(X)\), and conversely.   Thus
\[
        {}_mR^2_{\mathrm{nr}}\Gm(X)={}_mD^1(X).
\]
Substituting these identifications in \eqref{eq:Bockstein-short} proves the
proposition.
\end{proof}

\section{Torsion birational motives and normalized invariants}
\label{sec:torsion-birational}

This section records two applications of the fact that the functors
\(R^q_{\mathrm{nr}}\), for \(q>0\), are normalized birational motivic
invariants on smooth projective varieties.  We first recall the small amount of
terminology on birational Chow motives which is needed in the sequel. In the statements involving normalized birational Chow motives, we assume that the exponential characteristic of
\(k\) is invertible in \(R\).

\subsection{Birational and normalized Chow motives}

We write
\(\Choweff_R(k)\) for the covariant category of effective Chow motives
with coefficients in \(R\), and \(\heff(X)\) for the effective Chow motive
of \(X\in\SmPr(k)\).  The category of birational Chow motives
\(\Chowbir_R(k)\) is the quotient of \(\Choweff_R(k)\) obtained by
killing the Lefschetz motive; equivalently, when the exponential
characteristic of \(k\) is invertible in \(R\), it is the pseudo-abelian
envelope of \(\Choweff_R(k)\) modulo the ideal of morphisms factoring
through Tate twists \(M(1)\).  We denote the image of \(\heff(X)\) by
\(\hbir(X)\).  For connected smooth projective \(X,Y\), one has the standard
formula
\begin{equation}\label{eq:Chowbir-Hom}
        \Chowbir_R(k)(\hbir(X),\hbir(Y))
        \simeq \CH_0(Y_{k(X)})\otimes R .
\end{equation}
See \cite[Section~1.1]{KahnTorsionOrder} and 
\cite[Section~2.3]{SatoYamazaki}; the latter also recall the comparison of the
standard variants of \(\Chowbir_R\) when the exponential characteristic is
invertible in \(R\).

The \emph{normalized birational Chow category} \(\Chownor_R(k)\) is the
quotient of \(\Chowbir_R(k)\) by the ideal of morphisms factoring through
the unit object
\[
        \bbone_R=\hbir(\Spec k).
\]
We write \(\hnor(X)\) for the image of \(X\), and abbreviate
\[
        \Chownor_R(T,S)
        :=\Chownor_R(k)(\hnor(T),\hnor(S)).
\]
For connected smooth projective \(X,Y\), the description above gives
\begin{equation}\label{eq:Chownor-Hom}
        \Chownor_R(k)(\hnor(X),\hnor(Y))
        \simeq
        \Coker\bigl(\CH_0(Y)\otimes R
        \longrightarrow \CH_0(Y_{k(X)})\otimes R\bigr),
\end{equation}
where the map is induced by extension of scalars.  This is
\cite[Definition~1.4]{KahnTorsionOrder}; see also
\cite[Section~2.3, especially (2.4)]{SatoYamazaki}.

\begin{definition}\label{def:normalized-invariant}
Let \(R\) be a coefficient ring and let \(\mathcal C\) be an \(R\)-linear
additive category.  A contravariant functor
\(\Psi:\SmPr(k)^{\mathrm{op}}\to\mathcal C\) is called \emph{motivic} if it
factors through an \(R\)-linear additive functor
\(\Choweff_R(k)^{\mathrm{op}}\to\mathcal C\).  It is \emph{birational} if it
factors through \(\Chowbir_R(k)^{\mathrm{op}}\), and it is \emph{normalized} if
\(\Psi(\Spec k)=0\).
\end{definition}

If \(R\) is a coefficient ring in which the exponential characteristic of
\(k\) is invertible, then a functor
\(\SmPr(k)^{\mathrm{op}}\to R\text{-}\mathrm{Mod}\) is normalized, birational
and motivic if and only if it factors through
\(\Chownor_R(k)^{\mathrm{op}}\); see
\cite[Lemma~2.6]{SatoYamazaki}.

\begin{definition}\label{def:torsion-order}
Let \(A\) be an object of an additive category.  We say that \(A\) is
\emph{torsion} if there exists \(m>0\) such that
\(m\cdot\id_A=0\).  If such an \(m\) exists, the smallest one is called the
\emph{torsion order} of \(A\).

Following Kahn, a connected smooth projective variety \(X\) is said to have
finite torsion order if
\[
        \CH_0(X_{k(X)})\otimes\Q\simeq \Q .
\]
Equivalently, the class of the generic point \(\eta_X\) in
\[
        \Coker\bigl(\CH_0(X)\to \CH_0(X_{k(X)})\bigr)
\]
is torsion. The order of this class is Kahn's torsion order and is denoted
\(\Tor(X)\). If the above rational triviality condition fails, we set
\(\Tor(X)=0\). See \cite[Definition~1.5 and Lemma~1.6]{KahnTorsionOrder};
compare also \cite[Proposition~2.10 and Definitions~2.12--2.13]{SatoYamazaki}.

\end{definition}

\begin{remark} In the language of normalized birational Chow motives this is equivalent,
in the appropriate coefficient range, to \(\hnor(X)\) being a torsion object.
\end{remark}

\subsection{Derived unramified functors as normalized motivic invariants}

\begin{proposition}\label{prop:Rq-normalized-invariant}
Let \(F\in\HI_R(k)\) and let \(q>0\).  The functor
\[
        \SmPr(k)^{\mathrm{op}}\longrightarrow R\text{-}\mathrm{Mod},
        \qquad
        X\longmapsto R^q_{\mathrm{nr},R}F(X)
\]
is a normalized birational motivic invariant.
\end{proposition}

\begin{proof}
For \(X\in\SmPr(k)\), adjunction gives
\begin{equation}\label{eq:Rq-as-DMo-Hom}
        R^q_{\mathrm{nr},R}F(X)
        \simeq
        \DMo_R(k)\bigl(M^\circ(X)_R,R_{\mathrm{nr},R}(F[0])[q]\bigr),
\end{equation}
where \(M^\circ(X)_R=\nu_{\le0}M(X)_R\).  The comparison theorem between pure
birational Chow motives and triangulated birational motives
\cite[Theorem~4.2.2]{KahnSujathaTriangulated} identifies the assignment
\(X\mapsto M^\circ(X)_R\) on smooth projective varieties with the corresponding
birational Chow-motive functor.  Hence \eqref{eq:Rq-as-DMo-Hom} is functorial for
Chow correspondences and factors through \(\Chowbir_R(k)^{\mathrm{op}}\).  Thus
it is a birational motivic invariant in the sense of Definition
\ref{def:normalized-invariant}.

It remains to check normalization.  For \(X=\Spec k\), the motive \(M(X)_R=R\) is birational, so
\(\nu^{\ge1}M(X)_R=0\) by \cite[Proposition 1.2]{KahnSujathaDerived}.  The long
exact sequence \eqref{eq:KS-long}, i.e. \cite[Proposition 2.2]{KahnSujathaDerived}, then gives
\[
        R^q_{\mathrm{nr},R}F(k)
        \simeq H^q_{\Nis}(\Spec k,F)=0
        \qquad(q>0),
\]
because every Nisnevich covering of \(\Spec k\) has a section. Hence, the functor is
normalized.
\end{proof}

\begin{corollary}\label{cor:torsion-order-bound}
Let \(X/k\) be smooth projective and connected, and assume that
\[
        \CH_0(X_{k(X)})\otimes\Q\simeq\Q.
\]
  Then, for every \(F\in\HI(k)\) and
all \(q>0\),
\[
        \Tor(X)\cdot R^q_{\mathrm{nr}}F(X)=0.
\]
If \(k\) is algebraically closed of exponential characteristic \(p\), and if
\(S/k\) is a smooth projective surface with finite torsion order,
then for every \(F\in\HI_{\Z[1/p]}(k)\) and all \(q>0\),
\[
        \exp_p\bigl(\NS(S)_{\tors}\bigr)\cdot
        R^q_{\mathrm{nr},\Z[1/p]}F(S)=0,
\]
where \(\exp_p(A)\) denotes the exponent of \(A\otimes\Z[1/p]\).
\end{corollary}

\begin{proof}
By Proposition \ref{prop:Rq-normalized-invariant}, the functor
\(X\mapsto R^q_{\mathrm{nr}}F(X)\) is a normalized motivic birational invariant, 
and \(\Tor(X)\) kills the value of every such invariant on \(X\)
\cite[Lemma~1.6]{KahnTorsionOrder}.  This proves the first assertion.

For surfaces over an algebraically closed field,  the
prime-to-\(p\) torsion order 
\[
        \Tor_p(S)=\exp_p(\NS(S)_{\tors})
\]
\cite[Corollary~5.4]{KahnTorsionOrder}, and this gives the stated annihilator.
\end{proof}

\subsection{Detection by low unramified cohomology on torsion birational surfaces}

In this subsection \(k\) is algebraically closed of exponential characteristic 
\(p\), and we put \(R=\Z[1/p]\).  For a smooth scheme \(X/k\), Sato and
Yamazaki use the prime-to-\(p\) unramified cohomology groups
\[
        H^i_{\mathrm{ur}}(X)
        =\varinjlim_{(n,p)=1} H^0_{\mathrm{Zar}}(X,\mathcal H^i_n),
\]
where \(\mathcal H^i_n\) is the Zariski sheaf associated with
\(U\mapsto H^i_{\mathrm{\acute{e}t}}(U,\mu_n^{\otimes(i-1)})\); see
\cite[Section~2.5, especially (2.9)]{SatoYamazaki}.  These functors are
normalized, birational and motivic invariants in the sense recalled above, hence they
extend to \(\Chownor_R(k)^{\mathrm{op}}\).  For
\(f\in\Chownor_R(T,S)\), we write
\[
        H^i_{\mathrm{ur}}(f):H^i_{\mathrm{ur}}(S)\to H^i_{\mathrm{ur}}(T)
\]
for the induced map.

\begin{corollary}\label{cor:SatoYamazaki-detection}
Let \(S/k\) be a connected smooth projective surface with finite torsion
order, let \(T/k\) be smooth projective, and let
\[
        f\in\Chownor_R(T,S).
\]
If
$H^i_{\mathrm{ur}}(f)=0\ \text{for }i=1,2,$
then for every \(F\in\HI_R(k)\) and every \(q>0\),
\[
        R^q_{\mathrm{nr},R}F(f):
        R^q_{\mathrm{nr},R}F(S)
        \longrightarrow
        R^q_{\mathrm{nr},R}F(T)
\]
is zero.
\end{corollary}

\begin{proof}
By Proposition \ref{prop:Rq-normalized-invariant}, the functor
\(X\mapsto R^q_{\mathrm{nr},R}F(X)\) is normalized, birational and
motivic, and therefore factors through \(\Chownor_R(k)^{\mathrm{op}}\) by
\cite[Lemma~2.6]{SatoYamazaki}.  For such a surface
\(S\), the following are equivalent for
\(f\in\Chownor_R(T,S)\): \(f=0\), all normalized birational motivic
functors vanish on \(f\), and \(H^1_{\mathrm{ur}}(f)=H^2_{\mathrm{ur}}(f)=0\)
\cite[Theorem~7.3]{SatoYamazaki}.  Applying this theorem to the functor of
Proposition \ref{prop:Rq-normalized-invariant} proves the claim.
\end{proof}

\begin{remark}
Sato and Yamazaki also identify the first two unramified cohomology functors in
this setting with the usual first unramified cohomology and the Brauer group
functor, in the relevant prime to characteristic coefficients
\cite[Proposition~2.14]{SatoYamazaki}.  This is not used in the proof above, but
it explains why Corollary \ref{cor:SatoYamazaki-detection} is a detection result
by classical low-degree unramified invariants.
\end{remark}

\section*{Acknowledgements} The author thanks prof. Mikhail Bondarko for valuable comments and discussions.

\end{document}